\documentclass[12pt]{amsart}

\newif\ifdebug
 \debugfalse

\newif \iffig
\figfalse

\newif \iftable
\tablefalse

\usepackage{sfilip_package_settings}
\usepackage{sfilip_abbreviations}
\usepackage{sfilip_thm_style_long}
\usepackage{verbatim, amscd}
\usepackage{soul}
\newcommand{\ti}[1]{\tilde{#1}}
\newcommand{\vp}{\varphi}
\renewcommand{\leq}{\leqslant}
\renewcommand{\geq}{\geqslant}

\newcommand{\sfloor}[1]{\left\lfloor #1 \right\rfloor}

\NoIndentAfterEnv{theorem}
\NoIndentAfterEnv{proposition}
\NoIndentAfterEnv{corollary}
\NoIndentAfterEnv{lemma}
\NoIndentAfterEnv{problem}
\NoIndentAfterEnv{question}

\title[Nonlinearizable Embeddings]{Nonlinearizable embeddings of elliptic curves in rational surfaces}

\hypersetup{
	pdftitle={Nonlinearizable embeddings of elliptic curves in rational surfaces},	
	pdfauthor={Simion Filip, Valentino Tosatti},	
	pdfsubject={math},		
	pdfkeywords={MSC},	
	pdfnewwindow=true,		
	linkcolor=darkblue,		
	citecolor=darkred,		
	filecolor=darkblue,		
	urlcolor=darkblue,		
	pdfborder={0 0 0},
	breaklinks=true
}

\thanks{Revised \textsc{\today} }

\date{March 2026}

\author{
	Simion Filip
}
\address{
	\parbox{0.5\textwidth}{
		Department of Mathematics\\
		University of Chicago\\
		5734 S University Ave\\
		Chicago, IL 60637\\}
}
\email{{sfilip@math.uchicago.edu}}

\author{
	Valentino Tosatti
}
\address{
	\parbox{0.7\textwidth}{Courant Institute School of Mathematics, Computing, and Data Science\\
New York University\\
251 Mercer St\\
New York, NY 10012\\}
}
\email{{tosatti@cims.nyu.edu}}

\begin{document}

\begin{abstract}
We show that for any smooth cubic in $\bP^2$, there exists a dense $G_\delta$ set of configurations of $9$ distinct points such that blowing up $\bP^2$ at these $9$ points, the strict transform of the cubic is not linearizable and has nontorsion normal bundle.
This answers a problem raised by Ogus in 1975.
\end{abstract}

%
\maketitle
%
\noindent\hrulefill
\tableofcontents
\nointerlineskip
\noindent \hrulefill
\ifdebug
   \listoffixmes
\fi


\section{Introduction}
	\label{sec:introduction}

\subsection{Linearization of complex submanifolds}
Let $X$ be a complex manifold and $Y\subset X$ a compact complex submanifold, with normal bundle denoted by $N_{Y/X}$ and zero section $Z\subset N_{Y/X}$, and denote by $\vp:Y\to Z$ the natural isomorphism. We say that the embedding $Y\subset X$ is linearizable if $Y$ admits a holomorphic tubular neighborhood in $X$, i.e. if there are open sets $Y\subset U\subset X$ and $Z\subset V\subset N_{Y/X}$ and a biholomorphism $\Phi:U\to V$ such that $\Phi(Y)=Z$ and $\Phi|_Y=\vp$. A weaker notion, which will also play a role in our work, is that of formal linearizability, where we only require the existence of a formal isomorphism between the formal completions of $X$ along $Y$ and of $N_{Y/X}$ along $Z$, which extends $\vp$. We will always implicitly assume that all manifolds that we consider are connected.

While every submanifold embedding is linearizable in the $C^\infty$ category, thanks to the classical tubular neighborhood theorem in differential topology, this is not true in our holomorphic setting. Being linearizable imposes nontrivial restrictions on the embedding -- for example $TX|_Y$ must split holomorphically -- and it is thus easy to find examples of non-linearizable embeddings.
For instance, a compact complex submanifold of $\bP^n$ is linearizable if and only if it is a linear subspace \cite{morrow-rossi}.

\subsection{Curves inside surfaces}
From now on, we will work in the simplest nontrivial setting, which is the embedding of a smooth curve $C$ into a smooth compact surface $S$. The selfintersection $(C^2)\in\mathbb{Z}$ equals the degree of the normal bundle $N_{C/S}$.

The case $(C^2)<0$ is completely understood thanks to classical work of Grauert \cite{grauert}, which shows in particular that the embedding is always linearizable. On the other hand, when $(C^2)>0$ the embedding is in general far from being linearizable, as is the case for example for any $C\subset\bP^2$ of degree at least two. The case $(C^2)=0$ (i.e. $N_{C/S}$ topologically trivial) is thus ``borderline'' for the linearization problem, and it will be the topic of this paper.

For a smooth curve $C\subset S$ with $(C^2)=0$, if the genus of $C$ is zero then Savel'ev \cite{savelev} proved that the embedding is always linearizable. We are thus naturally led to the study of the linearization problem for the embedding $C\subset S$ of a smooth genus $1$ curve $C$ into a smooth compact surface $S$ with $(C^2)=0$. As for the case when the genus of $C$ is larger than $1$, not much is known about it, but see e.g. \cite{Mishustin, Thom}.

\subsection{Elliptic curves with degree zero normal bundle}
The study of the linearization problem for elliptic curves embedded in surfaces with degree zero normal bundle was initiated by Arnold in his influential 1976 work \cite{arnold}. There he proved (see \autoref{dioph} below) that the embedding is linearizable provided the normal bundle is Diophantine, an arithmetic condition which is of Lebesgue full measure. He also noted that when the normal bundle is torsion, then in general there are obstructions to linearization, see \autoref{recap} for more details. See also his textbook \cite[\S 27, \S 36]{arnold_book}.

Perhaps the simplest examples of embeddings of an elliptic curve with zero self-intersection inside a compact surface are obtained by the following procedure. Take a smooth cubic curve $E\subset\bP^2$ with $9$ distinct points on it, let $S\to\bP^2$ be the blowup at these $9$ points, and let $C\subset S$ be the strict transform of $E$. This setting was studied by Ogus \cite{Ogus} in 1975, at around the same time as Arnold, where he posed the following problem in \cite[Problem 4.16]{Ogus}:

\begin{problem}
Does there exist a configuration of $9$ distinct points on a smooth cubic curve $E\subset\mathbb{P}^2$ such that the embedding $C\subset S$ of the strict transform $C$ of $E$ into the blowup $S$ of $\mathbb{P}^2$ at these $9$ points has non-torsion normal bundle $N_{C/S}\in\mathrm{Pic}^0(E)$ and the embedding is not linearizable?
\end{problem}
This well-known problem was also recently reiterated in \cite[\S 1.5]{LTT} and in \cite[Question 1.8]{KoikeUehara2019_A-gluing-construction-of-K3-surfaces}, but there has been no progress since it was first posed.

Our main theorem gives in particular a positive answer to this problem:
\begin{theorem}\label{main}
For any smooth cubic curve $E\subset\mathbb{P}^2$, there exists a dense $G_\delta$ subset of the set of $9$ distinct point configurations on $E$ for which the conclusion of Ogus's problem holds: $N_{C/S}$ is non-torsion and the embedding $C\subset S$ is not linearizable.
\end{theorem}

About the ``dense $G_\delta$'' property: once a smooth cubic $E$ is given, the set of $9$ distinct point configurations on $E$ is an open subset $E^{(9)}\subset E^9$ of the Cartesian self-product, which we endow with its usual Euclidean topology.

\subsection{Formal Principles}
We now discuss an application of \autoref{main} to the formal principle. The formal principle, in the case of smooth curves inside surfaces, is the following property: given a smooth curve $C$ embedded in a smooth surface $S$, we say that the formal principle holds if whenever we have another such embedding $C'\subset S'$ and a formal isomorphism between the formal completions of $S$ along $C$ and of $S'$ along $C'$, then the germs of the two embeddings are biholomorphic. The papers \cite{grauert} and \cite{CG} proved that the formal principle holds whenever $(C^2)\neq 0$, but this is not the case when $(C^2)=0$. Indeed, if an embedding $C\subset S$ is formally linearizable but not linearizable then clearly the formal principle fails, and Arnold \cite{arnold} found germs of embeddings of an elliptic curve in a (noncompact) surface which have this property, thus showing that the formal principle can indeed fail (see also \cite{IP2}). The formal principle also makes sense in higher dimensions (and allowing the varieties to be singular), and the question of when exactly it is valid or not is a subject of much investigation, see e.g. \cite{griffiths}, \cite{hirschowitz},  \cite{kosarew}, \cite{steinbiss}, \cite{ABT}, \cite{hwang}, \cite{honghwang}, \cite{gongstolfp}, \cite{koike_new} and references therein.

A variant of the formal principle, called ``projective formal principle'', was recently introduced by Pereira--Thom in \cite[Defn.1.2]{PT}. Here one considers a compact smooth curve $C$ embedded in a projective surface $S$ (not necessarily smooth), and says that the projective formal principle holds for this embedding if whenever we have another such embedding $C'\subset S'$ (with $S'$ projective), and  a formal isomorphism between the formal completions of $S$ along $C$ and of $S'$ along $C'$, then the germs of the two embeddings are biholomorphic.

Up to now there were no known examples where the projective formal principle fails. As a direct consequence of \autoref{main}, we can show that such examples do exist:

\begin{corollary}\label{coro}
There exist embeddings $C\subset S$ of a smooth curve into a smooth projective surface which do not satisfy the projective formal principle.
\end{corollary}

\subsection{Ideas from the proofs}
We note first that a standard argument to disprove linearizability is via ``materialization of resonances'', see the discussion in \autoref{ssec:Materialization-of-resonances} below. However, this method does not work in the projective case, as we explain in \autoref{sssec:No-materialization-of-resonances-in-projective-case} that such a materialization never occurs.

\paragraph{Torsion normal bundle and fibrations}
The starting point for us is the well-known observation that when our embedding $C\subset S$ has torsion normal bundle, linearizability is completely understood: in this case $S$ has an elliptic fibration with $C$ as one of the fibers (possibly with multiplicity), known as Halphen pencils in the algebro-geometric literature, and linearizability holds if and only if this fibration is isotrivial. Since torsion points are dense in $\mathrm{Pic}^0(E)$, and since these isotrivial fibrations should be quite rare, one can reasonably expect to have a dense subset of $E^9$ ($9$-tuples of points on $E$) for which the corresponding embedding is nonlinearizable and has torsion normal bundle. We show that this expectation is indeed correct.

\paragraph{Formal linearization}
On the other hand, when the normal bundle of $C\subset S$ is nontorsion, Arnold proved that the embedding is formally linearizable, by first putting the embedding in a preliminary normal form (this step uses a deep theorem of Siu \cite{Siu}), and then constructing the linearization formally order by order. The fact that $N_{C/S}$ is nontorsion means that there is no division by zero when solving the corresponding homological equation. Of course, one cannot conclude from this that the formal linearization is convergent, due to the presence of small divisors.
In our setting, we first implement a family version of Arnold's setup, where we allow the $9$ points to vary in some small ball $B\subset E^9$, and obtain a family preliminary normal form, and from that a family formal linearization for those points in $B$ where the corresponding normal bundle is nontorsion.

\paragraph{Radius function}
The formal linearization has coefficients which are meromorphic in $B$ and may have poles at the torsion points. One key innovation then is to introduce a ``radius function'' $\rho:B\to[0,+\infty]$ which is upper semicontinuous, vanishes at those parameters where the formal linearization is divergent, and is strictly positive at a nontorsion point if and only if the formal linearization is convergent.
Given this, we conclude as follows: since the nonlinearizable torsion points are dense in $B$, as mentioned earlier, and since $\rho$ is upper semicontinuous, the set $\{\rho=0\}$ is a dense $G_\delta$. On the other hand, the torsion points are $F_\sigma$ with empty interior, by Baire, and so the nontorsion points where $\rho$ vanishes must also be a dense $G_\delta$, as desired.

As for the radius function, one would naively want to define it as the radius of convergence of the formal linearization. This does not quite work, since in general this function would not be upper semicontinuous. We end up defining $\rho$ first away from all resonances, by imitating the formula for the radius of convergence but replacing a limsup by a sup, and then taking its upper semicontinuous envelope across the resonances.

\paragraph{Analogy with 1d holomorphic maps}
There is a well-known analogy, going back to \cite{arnold}, between linearization of fixed points of holomorphic maps $z\mapsto \lambda z + O(z^2)$ with $|\lambda|=1$, and linearizability of elliptic curves with degree zero normal bundle.
We refer to \autoref{ssec:Materialization-of-resonances} for more on this, but note that our technique of using a radius function and semicontinuity to show the existence of nonlinearizable examples with nontorsion normal bundle is inspired by the usage of a similar method by Yoccoz in the Siegel problem for holomorphic maps, see \cite[Thm 11.14 and Lemma 11.15]{Milnor2006_Dynamics-in-one-complex-variable}.
Namely, allowing the parameter $\lambda$ to vary in the entire unit disk $|\lambda|\leq 1$, linearizability for $|\lambda|<1$ is automatic, while for $|\lambda|=1$ one can deduce linearizability in a set of full Lebesgue measure from the upper semicontinuity of the ``conformal radius'' function. Observe though that in our setting semicontinuity is used in the opposite direction, to prove nonlinearizability.

\paragraph{Paper outline}
The paper is organized as follows. In \autoref{sec2} we collect some preliminary results from the literature, explain Arnold's construction of elliptic curve neighborhoods via gluing maps, and we prove the existence of a preliminary normal form in families. In \autoref{sec3}, we upgrade the preliminary normal form to a family formal linearization, with meromorphic dependence on the parameter. The radius function is introduced in \autoref{sec4}, where its key properties are proved. Section \ref{sec5} contains algebro-geometric arguments involving Halphen pencils, to show that a dense subset of $E^9$ has corresponding embeddings which are nonlinearizable and have torsion normal bundle. The proof of \autoref{main} and \autoref{coro} are also given there. Lastly, in \autoref{sec:Complements-and-Questions} we formulate several natural questions, one related to the Brjuno condition, and another related to complete Ricci-flat \Kahler metrics in $S\backslash C$.


\subsection*{Acknowledgments}
We are grateful to Igor Dolgachev and Matthias Sch\"utt for many discussions about Halphen pencils, with special thanks to Matthias for providing us with the proof of \autoref{dense}. Thanks also to Serge Cantat, Yitwah Cheung, John Lesieutre, Keiji Oguiso, Ulf Persson, Fr\'ed\'eric Touzet, and Junsheng Zhang for useful communications. Serge Cantat also provided us with useful feedback on an earlier draft.
This material is based upon work partially supported by the National Science Foundation under Grants No. DMS-2005470, DMS-2305394 (first-named author) and DMS-2404599 (second-named author).


\section{A preliminary normal form}\label{sec2}

\subsection{Linearizability}
Let $Y\subset X$ a compact complex submanifold embedding, with normal bundle denoted by $N_{Y/X}$. In the introduction we defined this embedding to be linearizable if $Y$ admits a holomorphic tubular neighborhood in $X$. To explain the terminology, it is elementary to see that linearizability is equivalent to the following property: $Y$ can be covered by finitely many charts $U_i\subset X$, such that in each chart $U_i$ there are local holomorphic coordinates $(z^{(i)},w^{(i)}):=(z^{(i)}_1,\dots, z^{(i)}_k,w^{(i)}_1,\dots, w^{(i)}_{n-k})$ ($n=\dim X, k=\dim Y)$ with $U_i\cap Y=\{w^{(i)}_1=\cdots =w^{(i)}_{n-k}=0\}$ and for each $i\neq j$ the chart transition maps are given by
\begin{equation}
	(z^{(i)},w^{(i)})\mapsto (z^{(j)},\lambda_{ij}(z^{(i)},w^{(i)})\cdot w^{(i)}),
\end{equation}
where $\lambda_{ij}:U_i\cap U_j\to GL(n-k,\mathbb{C})$ is the cocycle defining $N_{Y/X}$.
\subsection{Useful facts from the literature}\label{recap}
Let $E\subset\mathbb{P}^2$ be a smooth cubic curve, $p=(p_1,\dots,p_9)$ be a configuration of $9$ distinct points on $E$, $\pi:S\to\mathbb{P}^2$ the blowup of $\mathbb{P}^2$ at $p$, and $\iota:C\hookrightarrow S$ the embedding of the strict transform of $E$ into $S$. Then $(C^2)=0$, so the normal bundle $N_{C/S}$ has degree zero.

Let us recall the following basic facts:

\begin{proposition}{\cite[Theorem 4.2.1]{arnold}}\label{formal}
Suppose $N_{C/S}\in\mathrm{Pic}^0(C)$ is not torsion. Then the embedding $C\subset S$ is formally linearizable.
\end{proposition}

\begin{proposition}{\cite[Lemma I.6.2 and Theorem II.5.1]{Neeman1989_Ueda-theory:-theorems-and-problems}}
		\label{neeman}
	The normal bundle $N_{C/S}\in\mathrm{Pic}^0(C)$ is torsion of order $\ell\geq 1$ if and only if there is an elliptic fibration $f:S\to \mathbb{P}^1,$ given by the linear system $|-\ell K_S|$, such that $\ell C$ equals the scheme-theoretic fiber of $f$ over some point in $\mathbb{P}^1$.
\end{proposition}

Such an elliptic fibration is called a Halphen pencil of index $\ell$. It is known (see e.g. \cite[Prop.2.2]{cantat_dolgachev}) that there is a birational morphism $\pi:S\to\mathbb{P}^2$. Projecting the elliptic fibration down to $\bP^2$, we obtain an irreducible pencil of curves of degree $3\ell$ in $\mathbb{P}^2$ passing through $9$ base points with multiplicity $\ell$, and with one fiber which is a cubic with multiplicity $\ell$, and $\pi$ is the blowup of these $9$ points.

\begin{proposition}{\cite[Proposition 5.1]{LTT}}\label{isotr}
Suppose $N_{C/S}\in\mathrm{Pic}^0(C)$ is torsion. Then the following are equivalent:
\begin{itemize}
\item[(1)] $C\subset S$ is linearizable,
\item[(2)] $C\subset S$ is formally linearizable,
\item[(3)] the elliptic fibration $f$ is isotrivial.
\end{itemize}
\end{proposition}

\begin{example}[Isotrivial Halphen Pencil]	
	There do exist smooth cubics $E$ with configurations of $9$ distinct points on $E$ such that the corresponding embedding $C\subset S$ is linearizable and has torsion normal bundle.
	Here is an explicit example.
	Consider the pencil of cubic curves
	\begin{equation}
		\lambda(x^3-y^3)+\mu(y^3-z^3)=0,\quad [\lambda:\mu]\in\bP^1.
	\end{equation}
	It has $9$ distinct base points, given by $[\zeta^p\zeta^q: \zeta^p:1], 1\leq p,q\leq 3,$ where $\zeta$ is a primitive third root of unity, and the pencil passes through each one with multiplicity $1$.
	This is thus a Halphen pencil of index $1$, and blowing up these $9$ points we get a rational surface $S$ with an elliptic fibration, whose general fibers are strict transforms of the members of the pencil, and with $3$ singular fibers of Kodaira type $IV$.
	Also, this pencil is isotrivial, since every smooth member is isomorphic to the Fermat cubic $x^3+y^3+z^3=0$ (the $j$-invariant is identically $0$). Taking $E$ to be a smooth member of this pencil, we obtain the desired example. The normal bundle of the strict transform $C$ in $S$ is trivial, and the embedding is linearizable since the elliptic fibration on $S$ is isotrivial.
\end{example}
	
The following fundamental result by Arnold \cite{arnold} will not be used directly in our proofs, see also \cite{GongStolovitch_Equivalence-of-neighborhoods-of-embedded-compact2022}, \cite{gongstol}, \cite{koikestol}, \cite{ogawa} for recent extensions:

\begin{theorem}{\cite{arnold}}\label{dioph}
Suppose $N_{C/S}\in\mathrm{Pic}^0(C)$ is Diophantine. Then the embedding $C\subset S$ is linearizable.
\end{theorem}
Here, a line bundle $L\in \mathrm{Pic}^0(C)$ is called Diophantine if there exist constants $A,B>0$ such that for all $n\geq 1$ we have
\begin{equation}
d(\mathcal{O}_C,L^n)\geq An^{-B},
\end{equation}
where $d$ is any fixed translation-invariant distance on $\mathrm{Pic}^0(C)$. The complement of Diophantine line bundles has Lebesgue measure zero in $\mathrm{Pic}^0(C)$, and has Hausdorff dimension zero.

\subsection{Preliminary normal form in families}
The main references for this section are \cite{arnold}, \cite[\S 27, \S 36 M--N]{arnold_book} and \cite{IP}.

\subsubsection{Uniformization and line bundles}
	\label{sssec:Uniformization-and-line-bundles}

Let $E$ be a smooth genus $1$ curve.
There exists $q\in\mathbb{C}$ with $|q|<1$, such that
$$E\isom \mathbb{C}^\times/q^{\mathbb{Z}}.$$
The total space of any degree zero line bundle $L\to E$ is biholomorphic to the quotient of $\mathbb{C}^\times\times\mathbb{C}\ni (z,r)$ by the identification
\begin{equation}\label{glue}
(z,r)\sim (qz,\lambda r),
\end{equation}
for some $\lambda\in\mathbb{C}^\times$.
Two such line bundles $L_1,L_2$ (with corresponding $\lambda_1,\lambda_2$) are isomorphic if and only if
\begin{equation}
	\lambda_1=\lambda_2q^k,
\end{equation}
for some $k\in\mathbb{Z}$.
The trivial line bundle therefore corresponds to taking $\lambda=1$ or a power of $q$ and more generally, a degree zero line bundle with parameter $\lambda$ is $\ell$-torsion in $\mathrm{Pic}(E)$ if and only if $\lambda^\ell=q^k$ for some $k\in\mathbb{Z}$.

\subsubsection{Patching}
	\label{sssec:Patching}

We now explain a construction of elliptic curves, as well as their neighborhoods in surfaces.
\autoref{siu} is the family analog of Arnold's result in \cite[Prop. 4.1.1]{arnold} (see also \cite[\S 27.F]{arnold_book}, \cite[p.109-110]{IP}, \cite[Lemma 2.1]{LTV}) which shows that any germ of an elliptic curve in a surface can be obtained in this manner.

Let us fix the notation for annuli in $\bC^\times$ and disks in $\bC$:
\begin{align}
	\label{eqn:annulus_def}
	\begin{split}
		A^\circ(r,R)& :=\{z\in\bC^\times: r<|z|<R\}\\
		A(r,R) & := \{z\in\bC^\times: r\leq |z|\leq R\}\\
		D^\circ(r) & := \{z\in\bC: |z|<r\}
	\end{split}
\end{align}
for $0<r<R$.
We will sometimes abuse notation and write $A(r,r)=\{z\in\bC^\times: |z|=r\}$ for the circle of radius $r$.

As a start, suppose $f\colon A^\circ(1-\ve,1+\ve)\to \bC^\times$ is a holomorphic function giving a biholomorphism between $A^\circ(1-\ve,1+\ve)$ and an open neighborhood of $|z|=|q|$ in $\bC^\times$.
Then we can ``glue'' the two neighborhoods and obtain an elliptic curve as follows.
We take as charts of the elliptic curve the open set $U_0$ between $A(1,1)$ and $f(A(1,1))$, another two more charts $U_1:=A^\circ(1-\ve,1+\ve)$ and $U_q:=f(A^\circ(1-\ve,1+\ve))$.
Then the gluing maps between $U_0$ and $U_1,U_q$ are given by the overlaps in $\bC^\times$, and the gluing map between $U_1$ and $U_q$ is given by $f$.
Note that no point belongs to three charts, provided $\ve$ is sufficiently small.
By the uniformization theorem, there exists a holomorphic map $g$ defined in some $A(|q|-\ve_1, 1+\ve_1)$ and $q'\in \bC^\times$ with $|q'|<1$ such that
\[
	q'\cdot g(z) = g(f(z)) \quad \forall z\in A(1-\ve_1,1+\ve_1).
\]
It is clear from this construction that only the germ of $f$ at $|z|=1$ matters.

\subsubsection{Patching a surface}
	\label{sssec:Patching-a-surface}

Suppose now given a map $f\colon A^\circ(1-\ve,1+\ve)\times D^\circ(\ve) \to \bC^\times \times D^\circ(1)$ with the property that $f(z,0)\in \bC^\times\times\{0\}$, and which is a biholomorphism between $A^\circ(1-\ve,1+\ve)\times D^\circ(\ve)$ and an open neighborhood of $|z|=|q|$ in $\bC^\times\times D^\circ(1)$.
Then we can glue the two neighborhoods and obtain a surface containing an elliptic curve as follows.
The open set $U_0\subset \bC^\times \times D^\circ(1)$ is a tubular neighborhood of the cylinder between $A(1-\ve_1,1-\ve_1)$ and $f(A(1-\ve_1,1-\ve_1)\times\{0\})$ (e.g. this set times a sufficiently small disk $D^\circ(\ve_2)$ in the transverse direction).
The other two open sets are tubular neighborhoods of $A(1-\ve_1,1+\ve)$ and $f(A(1-\ve_1,1+\ve)\times\{0\})$ respectively, and the gluing maps are given by the overlaps in $\bC^\times\times D^\circ(1)$ and by $f$.
We choose the $D^\circ(\ve_2)$ factor small enough (depending on $\ve_1$) so that no point belongs to three charts.

We note that this construction also works when the gluing map $f$ is allowed to vary in a holomorphic family (as will be discussed in more detail below).
We also note that two such patching constructions give biholomorphic surfaces if and only if, denoting by $f_1,f_2$ the gluing maps, there exists a biholomorphism onto its image $g\colon A(|q|-\ve_3,1+\ve_3)\times D^\circ(\ve_4) \to \bC^\times \times D^\circ(1)$ such that
\[
	g(f_1(z,r)) = f_2(g(z,r)) \quad \forall (z,r)\in A(1-\ve_3,1+\ve_3)\times D^\circ(\ve_4)
\]
where $\ve_3,\ve_4$ are sufficiently small but such that $A(|q|-\ve_3,1+\ve_3)$ contains $f_1(A(1,1)\times \{0\})$, and we have $g(z,0)\in \bC^\times\times\{0\}$.

\subsubsection{Uniformization in families}
	\label{sssec:Uniformization-in-families}
We now proceed to a discussion of these constructions in families, which will be the main setting for our proofs.

Suppose we have a ball $B\subset\mathbb{C}^N$, and a proper holomorphic submersion $h:\mathcal{E}\to B$ from an $(N+1)$-fold $\mathcal{E}$ to $B$ with connected fibers $E_w=h^{-1}(w), w\in B$, where $E_w\cong \mathbb{C}^\times/q(w)^{\mathbb{Z}}$ is a genus one curve, and $q\colon B\to\mathbb{C}$ is holomorphic with $|q(w)|<1$ for all $w\in B$.
Suppose we also have a proper holomorphic submersion $p\colon \mathcal{S}\to B$ from an $(N+2)$-fold $\mathcal{S}$ to $B$ with connected fibers $\mathcal{S}_w=p^{-1}(w), w\in B$, which are therefore compact complex surfaces. Lastly, suppose we also have a holomorphic embedding $\iota\colon \mathcal{E}\to \mathcal{S}$ such that $p\circ\iota=h$, which we can thus think of as a holomorphic family of embeddings $\iota_w\colon E_w\to \mathcal{S}_w,w\in B$.
We also assume that the normal bundle $N_{E_w/\mathcal{S}_w}$ has degree zero, i.e. $N_{E_w/\mathcal{S}_w}\in \mathrm{Pic}^0(E_w), w\in B$, and so $N_{E_w/\mathcal{S}_w}$ is obtained by the gluing in \autoref{glue} with $\lambda=\lambda(w)$ a holomorphic function $\lambda\colon B\to\mathbb{C}^\times$.
Lastly, in order for the discussion to be nontrivial, we will always assume that there is some value of $w\in B$ for which $N_{E_w/\mathcal{S}_w}$ is nontorsion.

\begin{proposition}[Preliminary normal form in families]
		\label{siu}
Up to shrinking $B$, there is an open neighborhood $\mathcal{U}$ of $\iota(\mathcal{E})$ in $\mathcal{S}$ and a biholomorphism
\begin{equation*}
	\Phi\colon \mathcal{U}\to \cU_{nf}
\end{equation*}
where $\cU_{nf}$ is a neighborhood of the central fiber in the manifold obtained as follows.
There exists holomorphic functions $a,b \colon B\times A^\circ(1-\ve,1+\ve) \times D^\circ(\ve) \to \bC$ such that $\cU_{nf}$ is obtained by patching as in \autoref{sssec:Patching-a-surface} via the gluing map
\begin{equation}
	\label{normalform}
	f(w,z,r):= \Big(w, q(w)z\big(1+r\cdot b(w,z,r)\big),\lambda(w) r \big(1 + r\cdot a(w,z,r)\big)\Big)
\end{equation}
\end{proposition}
Note that the resulting $\cU_{nf}$ is equipped with a proper holomorphic submersion ${\rm pr}_{B,nf}\colon \cU_{nf}\to B$ and we have
${\rm pr}_{B,nf}\circ\Phi={\rm pr}_{B}$.

\begin{proof}
To start, by a classical result of Kodaira-Spencer \cite[Theorem 14.3]{kodaira_spencer} (see also \cite[Satz 3.6]{wehler}), up to shrinking $B$, there is a biholomorphism $\Pi:(B\times\mathbb{C}^\times)/\mathbb{Z}\to \mathcal{E}$ with $h\circ\Pi={\rm pr}_B$, where the $\mathbb{Z}$-action is generated by $(w,z)\mapsto (w,q(w)z),$ $(w,z)\in B\times\mathbb{C}^\times$. This gives in particular a holomorphic $\mathbb{Z}$-covering map $\sigma:B\times\mathbb{C}^\times\to\mathcal{E}$.

Now, by assumption, each normal bundle $N_{E_w/\mathcal{S}_w}$ is topologically trivial, and by Ehresmann's Lemma the proper submersion $p:\mathcal{S}\to B$ is a smooth fiber bundle, so the normal bundle of $\iota(\mathcal{E})$ in $\mathcal{S}$ is also topologically trivial. We can thus find a neighborhood $\mathcal{V}$ of $\iota(\mathcal{E})$ in $\mathcal{S}$ and a homeomorphism $\Psi:\mathcal{V}\to \mathcal{E}\times D$, where $D\subset\mathbb{C}$ is a disc, with $(\Psi\circ\iota)(p)=(p,0)$ for all $p\in \mathcal{E}$. We can then consider the $\mathbb{Z}$-covering space $\pi:\ti{\mathcal{V}}\to \mathcal{V}$ which extends the $\mathbb{Z}$-covering $\sigma:B\times\mathbb{C}^\times\to \mathcal{E}$. There is a homeomorphism $\ti{\Psi}:\ti{\mathcal{V}}\to B\times \mathbb{C}^\times\times D$, such that the following commutes
\begin{equation}
\begin{tikzcd}
\ti{\mathcal{V}} \arrow{r}{\ti{\Psi}} \arrow[swap]{d}{\pi} &  B\times \mathbb{C}^\times\times D \arrow{d}{\sigma\times\mathrm{Id}} \\%
\mathcal{V} \arrow{r}{\Psi}& \mathcal{E}\times D
\end{tikzcd}
\end{equation}
and so $\ti{\Psi}^{-1}(B\times\mathbb{C}^\times\times\{0\})=\pi^{-1}(\iota(\mathcal{E}))$.
If we call $G: B\times \mathbb{C}^\times\times D\to  B\times \mathbb{C}^\times\times D$ the primitive deck transformation of the covering $\sigma\times\mathrm{Id}:B\times \mathbb{C}^\times\times D\to \mathcal{E}\times D$, so $G(w,z,r)=(w,q(w)z,r)$, then $F:=\ti{\Psi}^{-1}\circ G\circ \ti{\Psi}:\ti{\mathcal{V}}\to\ti{\mathcal{V}}$ is a primitive deck transformation for $\pi$.

The embedding $\iota:\mathcal{E}\to\mathcal{V}$ lifts to a holomorphic embedding $\ti{\iota}:B\times \mathbb{C}^\times\to \ti{\mathcal{V}}$ with topologically trivial normal bundle (with $\pi\circ\ti{\iota}=\iota$).
Since $B\times\mathbb{C}^\times$ is Stein this normal bundle is also holomorphically trivial, and a deep theorem of Siu \cite[Corollary 1]{Siu} shows that this embedding is linearizable, i.e. isomorphic to a neighborhood of the zero section inside its normal bundle.
More precisely, there exists an open neighborhood $\cN\subset B\times \bC^\times \times D$ of $B\times \bC^\times \times \{0\}$, an open neighborhood $\ti{\mathcal{U}}$ of $\ti{\iota}(B\times \mathbb{C}^\times)$ in $\ti{\mathcal{V}}$, and a biholomorphism
\begin{equation*}
	\ti{\Phi}\colon \ti{\mathcal{U}}\to \cN\subset  B\times\mathbb{C}^\times\times D,
\end{equation*}
with the property that $\ti{\Phi}^{-1}(B\times\mathbb{C}^\times\times\{0\})=\ti{\iota}(B\times \mathbb{C}^\times)$ and $\ti{\Phi}$ equals $\ti{\Psi}$ when restricted to this set.
If we let $\ti{F}:=\ti{\Phi}\circ F\circ\ti{\Phi}^{-1}$, we get two possibly smaller neighborhoods $\cN_1,\cN_2\subset B\times \mathbb{C}^\times\times D$ with $\ti{F}\colon \cN_1\to \cN_2$ biholomorphic, which equals the identity in the $B$ factor, and satisfies:
\begin{equation}
	\label{autom}
	\ti{F}(w,z,0)=(\ti{\Phi}\circ\ti{\Psi}^{-1}\circ G)(w,z,0)=(\ti{\Phi}\circ\ti{\Psi}^{-1})(w,q(w)z,0)=(w,q(w)z,0),
\end{equation}
for all $w\in B,z\in \mathbb{C}^\times$.
Write
\begin{equation*}
	\ti{F}(w,z,r)=\Big(w,B_1(w,z,r),A_1(w,z,r)\Big),
\end{equation*}
where $A_1,B_1$ are holomorphic on $\cN_1$.
Then \autoref{autom} implies that
\begin{equation*}
	B_1(w,z,r)=q(w)z(1+rb(w,z,r)),
\end{equation*}
for some holomorphic function $b$ on $\cN_1$.
Similarly, in the $D$ factor, $\ti{F}$ vanishes at $r=0$ (the zero section of the normal bundle), and its $r$-derivative at $r=0$ recovers the normal bundle $N_{E_w/\mathcal{S}_w}$ and so it equals $\lambda(w)$.
It follows that
\begin{equation*}
	A_1(w,z,r)=\lambda(w) r(1+ra(w,z,r)),
\end{equation*}
for some holomorphic function $a$ on $\cN_1$.
Altogether this gives
\begin{equation*}
	\ti{F}(w,z,r)=\Big(w,q(w)z\big(1+rb(w,z,r)\big),\lambda(w) r\big(1+ra(w,z,r)\big)\Big).
\end{equation*}
Up to shrinking again the base $B$, we can restrict the functions $A_1,B_1$ to an open set of the form $B\times A^\circ(1-\ve,1+\ve)\times D^\circ(\ve)$, and it follows that $\cU:=\pi(\ti{\cU})$ is biholomorphic with the fibration $\cU_{nf}$ obtained by patching as in \autoref{sssec:Patching-a-surface} via the gluing map $f$ given by \eqref{normalform}.
\end{proof}


\section{Formal linearization}\label{sec3}

\subsection{Formal linearization in families}\label{change}

\subsubsection{Series expansions}
	\label{sssec:Series-expansions}

Thanks to \autoref{siu}, we can choose a preliminary normal form as in \eqref{normalform} for our embedding.
This normal form is not unique, but we shall fix one such choice for the rest of the argument, so the holomorphic functions $a(w,z,r),b(w,z,r)$ are fixed.
Set
\begin{equation*}
	f(w,z,r)=\Big(w,q(w)z(1+rb(w,z,r)), \lambda(w)r(1+ra(w,z,r))\Big).
\end{equation*}
We seek a formal change of variable given by
\begin{equation}\label{g}
	g(w,z,r)=\Big(w,z(1+rd(w,z,r)), r(1+rc(w,z,r))\Big),
\end{equation}
satisfying the linearization property
\begin{equation}\label{conj}
	(g\circ f\circ g^{-1})(w,z,r)=\Big(w, q(w)z,\lambda(w)r\Big),
\end{equation}
where $c,d$ are given by formal power series in $r$
\begin{equation}\label{g2}
	d(w,z,r)=\sum_{n=0}^\infty d_n(w,z)r^n, \quad c(w,z,r)=\sum_{n=0}^\infty c_n(w,z)r^n,
\end{equation}
and the coefficients themselves have formal Laurent series expansions in $z$
\begin{equation}\label{laur}
c_n(w,z)=\sum_{k\in\mathbb{Z}}c_{n,k}(w)z^k, \quad d_n(w,z)=\sum_{k\in\mathbb{Z}}d_{n,k}(w)z^k,
\end{equation}
and the coefficients $c_{n,k}(w), d_{n,k}(w)$ are \emph{meromorphic} functions of $w\in B$.

\subsubsection{Resonances}
	\label{sssec:Resonances}
To describe the possible poles of the coefficients, for $n\in\mathbb{N}, k\in\mathbb{Z},$ set
\begin{equation}
	P_{n,k}:=\{w\in B\ |\ \lambda(w)^nq(w)^k=1\},
\end{equation}
be the set of resonances with exponents $n,k$.
Observe that the values of $w\in B$ for which we have a resonance $\lambda(w)^nq(w)^k=1$ for some $k\in\mathbb{Z}$ are exactly those for which $N_{E_w/\mathcal{S}_w}$ is $n$-torsion.
Since $\lambda^n q^k-1$ is a holomorphic function on $B$ which is not identically zero, thanks to our standing assumptions at the very beginning that there is some value of $w\in B$ for which $N_{E_w/\mathcal{S}_w}$ is nontorsion, we see that each $P_{n,k}$ is a (possibly empty) divisor in $B$.
In fact, more is true:

\begin{proposition}\label{discr}
	For each $n\in\mathbb{N},$ the subset $P_n:=\bigcup_k P_{n,k}\subset B$ is also a divisor, i.e. a locally finite union of closed irreducible analytic hypersurfaces.
\end{proposition}
\begin{proof}
For each $n$, if we have $\lambda(w)^nq(w)^k=1$ for some $k\in\mathbb{Z}$, then necessarily $k=\sfloor{-n\frac{\log|\lambda(w)|}{\log|q(w)|}}$ and as $w$ varies in some small open set, the quantity $\sfloor{-n\frac{\log|\lambda(w)|}{\log|q(w)|}}$ can achieve at most $2$ values. Thus, it suffices to consider these $2$ values of $k$, and the union of the set of zeros of $\lambda^n q^k-1$ for these $2$ values of $k$ is indeed a divisor.
\end{proof}
Clearly, the same conclusion holds for the subsets $ P_{\leq n}$, where we define
	\[P_{\leq n}:=\bigcup_{1\leq m\leq n}P_{m}.\]
For $w\in B$ we define also
\begin{equation*}
	A_w:=\{z\in\mathbb{C}\ |\ |q(w)|\leq |z|\leq 1\}.
\end{equation*}
The following is the main result of this section, and allows us to upgrade the preliminary normal form to a formal linearization, whose coefficients vary meromorphically in our family:
\begin{proposition}
For every $n\in\mathbb{N}, k\in\mathbb{Z}$ we can find unique meromorphic functions $c_{n,k}(w), d_{n,k}(w)$ with $w\in B$, such that for every $w\not\in  P_{\leq n}$ these functions are holomorphic at $w$ and the Laurent series for $c_n(w,z)$ and $d_n(w,z)$ in \eqref{g2} converge for $z$ in an open annulus containing $A_w$. For every such $w\not\in \bigcup_n P_{\leq n}$ and $z\in A_w$, the formal change of variable $g$ defined by \eqref{g}, \eqref{g2}, \eqref{laur} satisfies \eqref{conj}.
\end{proposition}
\begin{proof}
For this proof, let us say that a function $f(w,z)$ satisfies property $(\star_n)$ for some $n\geq 1$ if it has a Laurent series expansion $f(w,z)=\sum_{k\in\mathbb{Z}}f_{k}(w)z^k$ with $f_k(w)$ meromorphic functions in $w\in B$, such that if $w\not\in  P_{\leq n}$  then $f_k$ is holomorphic at $w$ for all $k\in\mathbb{Z}$, and $f(w,z)$ is holomorphic in $z$ in an open annulus containing $A_w$. We say that $f(w,z)$ satisfies the weaker property $(\dagger_n)$, for $n\geq 0$, if the same conditions hold except that $f(w,z)$ is only required to be holomorphic in $z$ in an open annulus containing the unit circle $\{|z|=1\}$ (rather than all of $A_w$); here we set $P_{\leq 0}:=\emptyset$. Both families are closed under taking sums, products, multiplication by holomorphic functions of $w$, and derivatives in the $z$ variable; moreover $(\star_n)$ implies $(\dagger_n)$, while $(\star_m)$ implies $(\star_n)$ and $(\dagger_m)$ implies $(\dagger_n)$ whenever $m\leq n$, since $P_{\leq m}\subset P_{\leq n}$ and $A_w$ contains $\{|z|=1\}$.

Furthermore, if $h(w,z)$ satisfies $(\star_n)$, then its dilation $h(w,q(w)z)$ satisfies $(\dagger_n)$. Indeed, for $w\not\in P_{\leq n}$ the function $h(w,\cdot)$ is holomorphic on an open annulus $U$ containing $A_w$, so $z\mapsto h(w,q(w)z)$ is holomorphic on $\{z\ |\ q(w)z\in U\}\supset\{z\ |\ q(w)z\in A_w\}=\{1\leq|z|\leq|q(w)|^{-1}\}$, which contains the unit circle; as $U$ is open, the region of holomorphy contains an open annulus about $\{|z|=1\}$. 

\noindent
{\bf The homological equation. }
Following Arnold we expand in convergent power series
\begin{equation}
rb(w,z,r)=\sum_{n=1}^\infty b_n(w,z)r^n, \quad ra(w,z,r)=\sum_{n=1}^\infty a_n(w,z)r^n,
\end{equation}
where the functions $a_n,b_n$ are holomorphic on $B\times A^\circ(1-\ve,1+\ve)$, and we construct the changes of variable $g_n$ inductively on $n\geq 1$ so as to solve \eqref{conj} ``one order at the time''.

Set $\Phi_0=\mathrm{id}$ and $f^{(0)}=f$, and suppose inductively that we have already constructed formal changes of variable $g_1,\dots,g_{n-1}$ as in \eqref{gn} below, in such a way that, setting $\Phi_{n-1}:=g_{n-1}\circ\cdots\circ g_1$, the partially normalized map
\begin{equation}\label{partialnf}
f^{(n-1)}:=\Phi_{n-1}\circ f\circ \Phi_{n-1}^{-1}
\end{equation}
has the form
\begin{equation}\label{partialnfform}
f^{(n-1)}(w,z,r)=\Big(w,\,q(w)z\big(1+\textstyle\sum_{m\geq n} b^{(n-1)}_m(w,z)\,r^m\big),\,\lambda(w) r\big(1+\sum_{m\geq n} a^{(n-1)}_m(w,z)\,r^m\big)\Big),
\end{equation}
i.e.\ all the terms of order $<n$ have been killed, and where the coefficients $a^{(n-1)}_m(w,z),b^{(n-1)}_m(w,z)$ satisfy property $(\dagger_{n-1})$ for every $m\geq n$. The base case $n=1$ holds with $\Phi_0=\mathrm{id}$, $a^{(0)}_m=a_m$ and $b^{(0)}_m=b_m$, which are holomorphic on $B\times A^\circ(1-\ve,1+\ve)$ and hence satisfy $(\dagger_0)$.

We then look for a formal change of variable
\begin{equation}\label{gn}
	g_n(w,z,r)=\Big(w,z\big(1+r^n\ti{d}_n(w,z)\big), r\big(1+r^n\ti{c}_n(w,z)\big)\Big),
\end{equation}
with $\ti{c}_n,\ti{d}_n$ functions to be determined that will satisfy $(\star_n)$, and we set $f^{(n)}:=g_n\circ f^{(n-1)}\circ g_n^{-1}$.
Ignoring terms involving $r^k, k>n$, we can write
\begin{equation*}
	g_n^{-1}(w,z,r)=(w,z(1-r^n\ti{d}_n(w,z)), r(1-r^n\ti{c}_n(w,z))),
\end{equation*}
and, using that $f^{(n-1)}$ has no terms of order $<n$ in \eqref{partialnfform},
\begin{equation}
	\begin{gathered}
		\label{conjn}
		f^{(n)}(w,z,r)=(g_n\circ f^{(n-1)}\circ g_n^{-1})(w,z,r)=\\
		=\Big(w,q(w)z\big(1+r^n(b^{(n-1)}_n(w,z)-\ti{d}_n(w,z)+\lambda(w)^n\ti{d}_n(w,q(w)z))\big),\\
		\lambda(w)r\big(1+r^n(a^{(n-1)}_n(w,z)-\ti{c}_n(w,z)+\lambda(w)^n\ti{c}_n(w,q(w)z))\big)\Big)+O(r^{n+1}),
	\end{gathered}
\end{equation}
so, in order to kill the order $r^n$ term, we need to choose $\ti{c}_n(w,z),\ti{d}_n(w,z)$ satisfying the homological equation
\begin{equation}\label{system}
\left\{\begin{aligned}
\lambda(w)^n\ti{c}_n(w,q(w)z)-\ti{c}_n(w,z)&=-a^{(n-1)}_n(w,z),\\
\lambda(w)^n\ti{d}_n(w,q(w)z)-\ti{d}_n(w,z)&=-b^{(n-1)}_n(w,z).
\end{aligned}\right.
\end{equation}
For this, we expand $a^{(n-1)}_n,b^{(n-1)}_n$ in convergent Laurent series in $z$
\begin{equation}\label{lau2}
a^{(n-1)}_n(w,z)=\sum_{k\in\mathbb{Z}}a^{(n-1)}_{n,k}(w)z^k, \quad b^{(n-1)}_n(w,z)=\sum_{k\in\mathbb{Z}}b^{(n-1)}_{n,k}(w)z^k,
\end{equation}
with $a^{(n-1)}_{n,k},b^{(n-1)}_{n,k}$ meromorphic in $w\in B$ and holomorphic for $w\not\in P_{\leq n-1}$ (by property $(\dagger_{n-1})$), and seek $\ti{c}_n,\ti{d}_n$ as analogous Laurent series
\begin{equation}\label{lau}
\ti{c}_n(w,z)=\sum_{k\in\mathbb{Z}}\ti{c}_{n,k}(w)z^k, \quad \ti{d}_n(w,z)=\sum_{k\in\mathbb{Z}}\ti{d}_{n,k}(w)z^k,
\end{equation}
which will be shown to satisfy $(\star_n)$.
Equating the $z^k$-terms in \eqref{system}, in order to satisfy \eqref{system} we need to have
\begin{equation}\label{system2}
\left\{\begin{aligned}
(\lambda(w)^nq(w)^k-1)\ti{c}_{n,k}(w)&=-a^{(n-1)}_{n,k}(w),\\
(\lambda(w)^nq(w)^k-1)\ti{d}_{n,k}(w)&=-b^{(n-1)}_{n,k}(w).
\end{aligned}\right.
\end{equation}
The unique solutions to equations \eqref{system2} are the meromorphic functions $\ti{c}_{n,k}(w),\ti{d}_{n,k}(w)$ of $w\in B$ given by
\begin{equation}\label{soln}
\ti{c}_{n,k}(w)=-\frac{a^{(n-1)}_{n,k}(w)}{\lambda(w)^nq(w)^k-1},\quad \ti{d}_{n,k}(w)=-\frac{b^{(n-1)}_{n,k}(w)}{\lambda(w)^nq(w)^k-1},
\end{equation}
which are holomorphic for $w\not\in P_{n,k}\cup P_{\leq n-1}$, hence in particular for $w\not\in P_{\leq n}$.

\noindent
{\bf Property $(\star_n)$ for $\ti{c}_n(w,z)$ and $\ti{d}_n(w,z)$. }
Next, we prove that the resulting functions $\ti{c}_n(w,z)$ and $\ti{d}_n(w,z)$ in \eqref{lau} satisfy $(\star_n)$, namely that given $w\not\in  P_{\leq n}$, then $\ti{c}_n(w,z),\ti{d}_n(w,z)$ are holomorphic in $z$ in an open annulus containing $A_w$.
Indeed, since $w\not\in P_{\leq n-1}$ and $a^{(n-1)}_n,b^{(n-1)}_n$ satisfy $(\dagger_{n-1})$, their Laurent series in \eqref{lau2} converge for all $z$ in an open annulus $A^\circ(1-\ve,1+\ve)$, for some $\ve>0$, and so
\begin{equation*}
	\limsup_{k\to+\infty}|a^{(n-1)}_{n,k}(w)|^{\frac{1}{k}}\leq 1-\ve, \quad \limsup_{k\to+\infty}|a^{(n-1)}_{n,-k}(w)|^{\frac{1}{k}}\leq 1-\ve,
\end{equation*}
and analogously for $b^{(n-1)}_n$.
Recalling \eqref{soln} (and the fact that $\ti{c}_{n,k}(w),\ti{d}_{n,k}(w)$ are holomorphic at $w$ for all $k\in\mathbb{Z}$), and using that $|q(w)|<1$, we have then
\begin{align*}
	\limsup_{k\to+\infty}\left|\frac{a^{(n-1)}_{n,k}(w)}{\lambda(w)^nq(w)^k-1}\right|^{\frac{1}{k}} & \leq 1-\ve,\\
	\limsup_{k\to+\infty}\left|\frac{a^{(n-1)}_{n,-k}(w)}{\lambda(w)^nq(w)^{-k}-1}\right|^{\frac{1}{k}}& \leq |q(w)|-\ve,
\end{align*}
which proves our claim that $\ti{c}_n(w,z)$ and $\ti{d}_n(w,z)$ indeed satisfy $(\star_n)$.

Finally, by the choice of $\ti{c}_n,\ti{d}_n$ the order $r^n$ term of $f^{(n)}$ in \eqref{conjn} vanishes, while its terms of order $<n$ remain zero; thus $f^{(n)}$ again has the form \eqref{partialnfform}, with $n$ replaced by $n+1$. It remains to check that its coefficients $a^{(n)}_m,b^{(n)}_m$ ($m\geq n+1$) satisfy $(\dagger_n)$, which is the inductive hypothesis at the next stage. Expanding $f^{(n)}=g_n\circ f^{(n-1)}\circ g_n^{-1}$ in powers of $r$ and Taylor-expanding $\ti{c}_n,\ti{d}_n$ in the $z$ variable, exactly as in the order $n$ computation \eqref{conjn}, we see that each $a^{(n)}_m,b^{(n)}_m$ is a finite sum of products, each factor of which is a holomorphic function of $w$ times one of: a coefficient $a^{(n-1)}_{m'}(w,z),b^{(n-1)}_{m'}(w,z)$ (with $n\leq m'\leq m$) or a $z$-derivative thereof; a function $\ti{c}_n(w,z),\ti{d}_n(w,z)$ or a $z$-derivative thereof; or a dilation $\ti{c}_n(w,q(w)z),\ti{d}_n(w,q(w)z)$ or such a dilated $z$-derivative. Here the coefficients of $f^{(n-1)}$ are evaluated at the base point $z$, through the inner map $g_n^{-1}$, and so are not dilated, while the dilation by $q(w)$ enters only through the outer composition with $g_n$, just as the term $\lambda(w)^n\ti{d}_n(w,q(w)z)$ arose in \eqref{conjn}. Now the $a^{(n-1)}_{m'},b^{(n-1)}_{m'}$ satisfy $(\dagger_{n-1})$, hence $(\dagger_n)$; the functions $\ti{c}_n,\ti{d}_n$ and their $z$-derivatives satisfy $(\star_n)$, hence $(\dagger_n)$; and their dilations by $q(w)$ satisfy $(\dagger_n)$ by the closure property established above. Since $(\dagger_n)$ is preserved under sums, products, and multiplication by holomorphic functions of $w$, it follows that $a^{(n)}_m,b^{(n)}_m$ satisfy $(\dagger_n)$ for all $m\geq n+1$. This completes the inductive step.

\noindent
{\bf The formal linearization. }
Repeating this for all $n$, we obtain a sequence of formal change of variables $g_n(w,z,r)$ as in \eqref{gn}. By construction $f^{(n)}=\Phi_n\circ f\circ \Phi_n^{-1}$, with $\Phi_n=g_n\circ\cdots\circ g_1$, has all its terms of order $\leq n$ equal to those of the linear map $(w,q(w)z,\lambda(w)r)$, so the desired formal change of variable $g(w,z,r)$ as in \eqref{g}, \eqref{g2} satisfying  \eqref{conj} is obtained as the limit
\begin{equation}\label{compos}
g=\lim_{n\to+\infty}g_n\circ g_{n-1}\circ\cdots \circ g_1,
\end{equation}
which, as we shall see presently, is well-defined because for each $j$ the coefficients of $r^j$ in the formal power series expansion of the components of $g_n\circ g_{n-1}\circ\cdots \circ g_1$ stabilize as $n$ grows.

To conclude the proof, we also need to show that the resulting functions $c_n(w,z)$ and $d_n(w,z)$ in \eqref{g}, \eqref{g2} satisfy property $(\star_n)$, for all $n\geq 1$.

\noindent
{\bf Property $(\star_n)$ for $c_n(w,z)$ and $d_n(w,z)$. }
To see these, for each $p\geq 1$ write
\begin{equation*}
	(g_p\circ \cdots \circ g_1)(w,z,r)=
	\Big(w,z\big(1+r\hat{d}^{(p)}(w,z,r)\big),r\big(1+r\hat{c}^{(p)}(w,z,r)\big)\Big),
\end{equation*}
and expand in formal power series
\begin{equation*}
r\hat{c}^{(p)}(w,z,r)=\sum_{n=1}^\infty \hat{c}^{(p)}_n(w,z)r^n, \quad r\hat{d}^{(p)}(w,z,r)=\sum_{n=1}^\infty \hat{d}^{(p)}_n(w,z)r^n.
\end{equation*}
We then claim that for each $p\geq 1$ we have
\begin{equation}\label{stabilize}
	\hat{c}^{(p+1)}_n(w,z)=\hat{c}^{(p)}_n(w,z), \quad \hat{d}^{(p+1)}_n(w,z)=\hat{d}^{(p)}_n(w,z),\quad\text{for }1\leq n\leq p,
\end{equation}
which is the precise stabilization statement, hence the formal power series in \eqref{g} are well-defined and we have
\begin{equation}
c_n(w,z)=\hat{c}^{(p)}_n(w,z), \quad d_n(w,z)=\hat{d}^{(p)}_n(w,z),\quad\text{for }1\leq n\leq p,
\end{equation}
and we also claim that for every $p\geq 1$ and every $n\geq 1$, the functions $\hat{c}^{(p)}_n(w,z),\hat{d}^{(p)}_n(w,z)$
can be expressed as finite sums and products of the functions $\ti{c}_m(w,z), \ti{d}_m(w,z), m\leq n$ and their derivatives in the $z$ variable. Together, these claims show that the functions $c_n(w,z), d_n(w,z)$ indeed satisfy property $(\star_n)$ as desired.

We will prove these two claims by induction on $p\geq 1$. The base case holds because
\begin{equation}
r\hat{c}^{(1)}(w,z,r)=rc_1(w,z),\quad r\hat{d}^{(1)}(w,z,r)=rd_1(w,z).
\end{equation}

For the induction step, we have
\begin{equation}\begin{gathered}
(g_{p+1}\circ g_p\circ \cdots \circ g_1)(w,z,r)=g_{p+1}(w,z(1+r\hat{d}^{(p)}(w,z,r)),r(1+r\hat{c}^{(p)}(w,z,r)))\\
=(w,z(1+r\hat{d}^{(p)}(w,z,r))(1+r^{p+1}(1+r\hat{c}^{(p)}(w,z,r))^{p+1}\ti{d}_{p+1}(w,z(1+r\hat{d}^{(p)}(w,z,r)))),\\
\ \ \ \ \ r(1+r\hat{c}^{(p)}(w,z,r))(1+r^{p+1}(1+r\hat{c}^{(p)}(w,z,r))^{p+1}\ti{c}_{p+1}(w,z(1+r\hat{d}^{(p)}(w,z,r))))),
\end{gathered}
\end{equation}
so that
\begin{equation}\begin{gathered}
1+r\hat{c}^{(p+1)}(w,z,r)=\\
=(1+r\hat{c}^{(p)}(w,z,r))(1+r^{p+1}(1+r\hat{c}^{(p)}(w,z,r))^{p+1}\ti{c}_{p+1}(w,z(1+r\hat{d}^{(p)}(w,z,r)))),
\end{gathered}
\end{equation}
\begin{equation}\begin{gathered}
	1+r\hat{d}^{(p+1)}(w,z,r)= \\
	=(1+r\hat{d}^{(p)}(w,z,r))(1+r^{p+1}(1+r\hat{c}^{(p)}(w,z,r))^{p+1}\ti{d}_{p+1}(w,z(1+r\hat{d}^{(p)}(w,z,r)))).
\end{gathered}\end{equation}
Expanding both sides in formal power series in $r$, using the expansion
\begin{equation}
	\ti{c}_{p+1}(w,z(1+r\hat{d}^{(p)}(w,z,r)))=\ti{c}_{p+1}(w,z)+\sum_{\ell\geq 1}\frac{\partial^\ell_z\ti{c}_{p+1}(w,z)}{\ell!}(zr\hat{d}^{(p)}(w,z,r))^\ell,
\end{equation}
and the analogous one for $\ti{d}_{p+1}$, and equating terms with the same powers of $r$ we see that \eqref{stabilize} indeed holds, and also that
for any $n\geq p+1$, the functions $\hat{c}^{(p+1)}_{n}(w,z),\hat{d}^{(p+1)}_{n}(w,z)$ can be expressed as finite sums and products of the functions $\ti{c}_{p+1}(w,z), \ti{d}_{p+1}(w,z)$ and their derivatives in the $z$ variable, as well as $\hat{c}^{(p)}_m(w,z),\hat{d}^{(p)}_m(w,z), m\leq n,$. Using the induction hypothesis, this proves the desired claims, and completes the proof of the Proposition. \end{proof}

\section{Convergence and the radius function}\label{sec4}
\subsection{The radius function}
Given a fixed preliminary normal form as in \autoref{siu}, in \autoref{sec3} we have constructed uniquely a formal change of variable
\begin{equation}\label{linear}
g(w,z,r)=(w,z(1+rd(w,z,r)), r(1+rc(w,z,r))),
\end{equation}
as in \autoref{change}, so that the linearization property \eqref{conj} holds, where $c(w,z,r),d(w,z,r)$ are formal power series in $r$
\begin{equation}
d(w,z,r)=\sum_{n=0}^\infty d_n(w,z)r^n, \quad c(w,z,r)=\sum_{n=0}^\infty c_n(w,z)r^n,
\end{equation}
and the functions $c_n,d_n$ have the Laurent expansions in \eqref{laur} with meromorphic coefficients $c_{n,k}, d_{n,k}$, such that for every $w\not\in P_{\leq n}$ the functions $c_n(w,z),d_n(w,z)$ are holomorphic in $z$ in an open annulus containing $A_w=\{|q(w)|\leq |z|\leq 1\}$. Up to shrinking $B$, we can fix a closed annulus $A$ such that $A_w\subset A^\circ$ for all $w\in B$, and for every $w\not\in P_{\leq n}$ the function $c_n(w,z),d_n(w,z)$ are holomorphic in $z$ in an open annulus containing $A$.


\subsubsection{The radius function}
	\label{sssec:The-radius-function}

The key object we will need is the following ``radius function'' $\rho:B\to[0,+\infty]$.
Recall that $P_{\leq n}$ is the locally finite divisor of possible poles for the coefficients $c_n(w,z), d_n(w,z)$, and we set $P:=\bigcup_n P_{\leq n}$, which need not be locally finite anymore.
We first define $\rho$ for $w\in B\backslash P$ by
\begin{equation}\begin{split}
	\rho(w):=\min\bigg\{\left(\sup_{n\geq 1}\sup_{z\in A}|c_n(w,z)|^{\frac{1}{n}}\right)^{-1},\left(\sup_{n\geq 1}\sup_{z\in A}|d_n(w,z)|^{\frac{1}{n}}\right)^{-1}\bigg\},
\end{split}
\end{equation}
with the convention that $\frac{1}{+\infty}=0, \frac{1}{0}=+\infty$. From \autoref{discr} we see that $ P\subset B$ is a countable union of divisors, hence in particular it has empty interior. For any $w\in  P$ we then define
\begin{equation}\label{extend}
	\rho(w):=\lim_{r\downarrow 0} \sup_{|w'-w|<r, w'\in B\backslash P}\rho(w').
\end{equation}

\subsection{Upper semicontinuity of the radius function and convergence}
The following are the key properties of the radius function $\rho$. We equip $B$ with its standard complex manifold topology.

\begin{proposition}\label{properties}
The function $\rho:B\to[0,+\infty]$ satisfies that:
\begin{itemize}
	\item[(a)] $\rho$ is upper semicontinuous,
	\item[(b)] Given $w\in B\backslash P$, we have $\rho(w)>0$ if and only if the embedding $\iota_w:E_w\to \mathcal{S}_w$ is linearizable,
	\item[(c)] Given $w\in P$, if $\rho(w)>0$ then the embedding $\iota_w:E_w\to \mathcal{S}_w$ is linearizable.
\end{itemize}
\end{proposition}
\begin{proof}
(a) First, we show that for each $n$ the function given by $w\mapsto \sup_{z\in A}|c_n(w,z)|^{\frac{1}{n}}$ is lower semicontinuous on $B\backslash  P_{\leq n}$. Indeed, given a point $w_0$ there, thanks to \autoref{discr} we have that $w\not\in  P_{\leq n}$ for all $w$ close to $w_0$, hence for all such $w$ the Laurent series defining $c_n(w,z)$ converges for $z\in A$ and so $z\mapsto |c_n(w,z)|$ is a continuous function on $A$. Since the pointwise supremum of a family of lower semicontinuous functions is also lower semicontinuous, this proves this claim.

It follows that $w\mapsto \left(\sup_{n\geq 1}\sup_{z\in A}|c_n(w,z)|^{\frac{1}{n}}\right)^{-1}$ is upper semicontinuous on $B\backslash P$, and hence so is $\rho$. From its definition, and using that $P\subset B$ has empty interior, it is then elementary to check that $\rho$ is upper semicontinuous on all of $B$ (the extension in \eqref{extend} is a standard construction in real analysis, known as usc extension or usc envelope). This proves part (a).

For parts (b) and (c), suppose that $R:=\rho(w)>0$ for some $w\in B$, and let us first assume that $w\not\in  P$. Then for each $n\in\mathbb{N}$ the functions $c_n(w,z)$ and $d_n(w,z)$ are holomorphic in $z$ in an open annulus containing $A$, and for all $z$ in this annulus we have
\begin{equation}
\left(\sup_{n\geq 1}\sup_{z\in A}|c_{n}(w,z)|^{\frac{1}{n}}\right)^{-1}\geq R,\quad
\left(\sup_{n\geq 1}\sup_{z\in A}|d_{n}(w,z)|^{\frac{1}{n}}\right)^{-1}\geq R.
\end{equation}
Since $\limsup_{n}\leq \sup_{n}$, these imply that the power series
\begin{equation}\label{power}
c(w,z,r)=\sum_{n\geq 0}c_n(w,z)r^n, \quad d(w,z,r)=\sum_{n\geq 0}d_n(w,z)r^n,
\end{equation}
converge on the disc of radius $R$, uniformly as $z$ varies in $A$. Thus, for this value of $w$, the formal linearization in \eqref{linear} is an analytic linearization.

If on the other hand we have $w\in  P$, then by the definition of $\rho$ 
we can find a sequence $w_i\to w$, $w_i\not\in  P$ such that $\rho(w_i)\geq \frac{R}{2}$ for all $i$, hence
\begin{equation}
\left(\sup_{n\geq 1}\sup_{z\in A}|c_{n}(w_i,z)|^{\frac{1}{n}}\right)^{-1}\geq \frac{R}{2},\quad
\left(\sup_{n\geq 1}\sup_{z\in A}|d_{n}(w_i,z)|^{\frac{1}{n}}\right)^{-1}\geq \frac{R}{2},
\end{equation}
hold for all $i$.
As above, these imply that the power series for $c(w_i,z,r)$ and $d(w_i,z,r)$ converge on the disc of radius $\frac{R}{2}$ uniformly as $z$ varies in $A$ and uniformly in $i$.
Passing to a subsequence in $i$, we can thus take uniform limits $\lim_i c(w_i,z,r), \lim_i d(w_i,z,r)$ to get holomorphic functions on $z\in A,$ and $r\in D^\circ(R/2)$ which give us an analytic linearization for this value of $w$.

Lastly, if $w\not\in P$ and if the embedding $\iota_w$ is linearizable, then there is a change of variable $\ti{g}(w,z,r)$, holomorphic in $z,r$, which applied to the preliminary normal form in \autoref{siu} linearizes it, namely
\begin{equation}\label{conjj}
(\ti{g}\circ f\circ \ti{g}^{-1})(w,z,r)=(w, q(w)z,\lambda(w)r).
\end{equation}
Since the formal change of variable $g$ in \eqref{linear} satisfies the same linearization equation \eqref{conj}, and since we have seen that the formal solution of this is unique, it follows that $\ti{g}=g$, and so $g(w,z,r)$ is holomorphic in $z,r$. In particular, there is some $R>0$ such that for all $z$ in the annulus $A$, the corresponding power series in \eqref{power} converges on the disc of radius $R$. Cauchy's estimates then provide us bounds for the coefficients $c_{n}(w,z), d_{n}(w,z)$ of the form
\[|c_n(w,z)|^{\frac{1}{n}}\leq C, \quad |d_n(w,z)|^{\frac{1}{n}}\leq C,\]
for a constant $C$ independent of $n$ and $z$, which imply that $\rho(w)\geq R'>0$ for some $R'>0$.
\end{proof}

\begin{remark}[Radius of convergence and $\rho$]
	\label{rmk:Radius-of-convergence-and-rho}
	We recall that the radius of convergence of a power series is defined using a $\limsup$, whereas our definition of $\rho$ uses a $\sup$.
	In general, the radius of convergence is not an upper semicontinuous function of parameters, whereas the function $\rho$ is, and importantly, vanishing of $\rho$ is equivalent to vanishing of the radius of convergence.

	The following simple example:
	\begin{equation}
		f_t(z)=\sum_{n=0}^\infty e^{n^2 t}z^n, \quad t\in\mathbb{C},
	\end{equation}
	exhibits a formal power series whose coefficients depend holomorphically in $t$ but whose radius of convergence is not an upper semicontinuous function of $t$.
	Indeed $f_0$ has radius of convergence $1$ while $f_{t}$ has infinite radius of convergence whenever $\Re t<0$.
\end{remark}
	
\section{Proof of the main theorem}\label{sec5}
\subsection{Isotrivial Halphen pencils}

\subsubsection{The family picture}
	\label{sssec:setup_the_universal_family}
Fix a smooth cubic curve $E\subset \bP^2$.
Denote by $E^{(9)}$ the open subset of $E^{9}$ consisting of $9$-tuples of distinct points.
For any $p\in E^{(9)}$ denote the blowup of $\bP^2$ at $p$ by
\[
	\cS_p:=Bl_p \bP^2 \text{ and by }\cE_p\text{ the strict transform of }E.
\]
Note that by construction the normal bundle has degree zero: $\deg N_{\cE_p/\cS_p}=0$.

\subsubsection{Line bundles}
	\label{sssec:line_bundles}
Recall that $K_{\bP^2}\vert_E$, the canonical bundle of $\bP^2$ restricted to $E$, has degree $-9$.
Also, $9$ distinct points on $E$ determine naturally a degree $9$ line bundle, so we can consider the composition
\begin{equation}
	\label{eqn_cd:nine_points_to_Pic0}
\begin{tikzcd}[column sep = large]
	E^{(9)}
	\arrow[r, ]
	\arrow[rr, bend right, "\chi"]
	&
	\Pic^9(E)
	\arrow[r, "\otimes K_{\bP^2}\vert_E"]
	&
	\Pic^0(E)
\end{tikzcd}
\end{equation}
Then $\chi$ is an open map from an open subset of $E^9$ to the complex $1$-dimensional torus $\Pic^0(E)$.

If $\pi:\mathcal{S}_p\to\mathbb{P}^2$ denotes the blowup map, then $\pi|_{\mathcal{E}_p}:\mathcal{E}_p\to E$ is an isomorphism and we will use the induced pullback map to identify $\mathrm{Pic}^0(\mathcal{E}_p)$ with $\mathrm{Pic}^0(E)$. Under this identification, we have that
\begin{equation}
N_{\cE_p/\cS_p}=\mathcal{O}_{\mathbb{P}^2}(3)|_E \otimes\mathcal{O}_E\left(-\sum_{j=1}^9 p_j\right)=-\chi(p).
\end{equation}

\subsubsection{Isodesmic divisors}
	\label{sssec:isodesmic_divisors}
For any $L\in \Pic^0(E)$ denote by $D_L:=\chi^{-1}(L)\subset E^{(9)}$, which is visibly a smooth divisor in $E^{(9)}$.
Furthermore, as $L$ ranges over all torsion line bundles in $\Pic^0(E)$, the set
\[
	\mathcal{T}:=\bigcup_{\substack{L\in \Pic^0(E)\\ \text{torsion}}}\hspace{-1em} D_L \hspace{1em}\text{ is dense in }E^{(9)}.
\]
We can write
\[
	\mathcal{T}=\bigcup_{n\geq 1} \mathcal{T}_n,\quad \mathcal{T}_n:=\bigcup_{\substack{L\in \Pic^0(E)\\ n\text{-torsion}}}\hspace{-1em} D_L,
\]
and each $\mathcal{T}_n$ is a finite union of complex hypersurfaces, so it is closed and with empty interior in $E^{(9)}$.
Recall that if $p\in D_L$ is a $9$-point configuration with $L$ torsion, and $\mathcal{S}_p\to\mathbb{P}^2$ is the blowup of $p$, then \autoref{neeman} gives us an elliptic fibration $f_p:\mathcal{S}_p\to\mathbb{P}^1$. The following will be a crucial ingredient for our main theorem:

\begin{proposition}\label{dense}
If $E\subset \mathbb{P}^2$ is a cubic curve then the set
\[\mathcal{T}_{\rm ni}:=\left\{p\in \bigcup_{\substack{L\in \Pic^0(E)\\ \text{torsion}}}\hspace{-1em} D_L\ \bigg|\ f_p\text{ is not isotrivial}\right\}\subset\mathcal{T}\subset E^{(9)}\]
is dense in $E^{(9)}$.
\end{proposition}
The proof below was communicated to us by Matthias Sch\"utt, whom we are very grateful to, and it improved our original argument which only applied to a very general cubic.
\begin{proof}
For each $n\geq 1$, we have that $\mathcal{T}_n$ is an $8$-dimensional quasiprojective variety. Let $\mathcal{T}_n^{\rm iso}\subset\mathcal{T}_n$ be the configurations $p\in\mathcal{T}_n$ such that the corresponding elliptic fibration $f_p$ is isotrivial. The main claim is then that for every $n\geq 1,$ we have that $\mathcal{T}_n^{\rm iso}$ is a closed algebraic subvariety of $\mathcal{T}_n$ with positive codimension. Given this, it follows right away that
\begin{equation}
\mathcal{T}_{\rm ni}=\bigcup_{n\geq 1}(\mathcal{T}_n\backslash\mathcal{T}_n^{\rm iso})
\end{equation}
is still dense in $E^{(9)}$, as desired.

Now the proof of the claim. Assume first that $n\geq 2$ (which is anyway sufficient to prove the above density statement). Given any $p\in \mathcal{T}_n$, we have the corresponding elliptic fibration $f_p:\mathcal{S}_p\to\mathbb{P}^1$, which contains $nE$ as a fiber (say over $q\in\mathbb{P}^1$). Let $\mathrm{Jac}(f_p):\mathcal{S}'_p\to\mathbb{P}^1$ be its Jacobian fibration, which is a rational elliptic surface with a section. Then, as mentioned for example in \cite[Rmk 3.2]{DLPU}, there exists an $n$-torsion point on the fiber $\mathrm{Jac}(f_p)^{-1}(q)\cong E$ such that performing a logarithmic transformation to $\mathcal{S}'_p$ along this fiber with this choice of torsion point recovers the original elliptic fibration $f_p:\mathcal{S}_p\to\mathbb{P}^1$. Furthermore, $f_p$ is isotrivial if and only if $\mathrm{Jac}(f_p)$ is. See also \cite[\S 1.6]{FM} for background on logarithmic transformations.

There is an $8$-dimensional quasiprojective parameter space $\mathcal{R}_n$ for rational elliptic surfaces with a section and with a marked fiber isomorphic to $E$, following \cite{Miranda}, but then also taking the marked fibre into account by way of some M\"obius transformation. Mapping $f_p$ to $\mathrm{Jac}(f_p)$ gives a morphism $\mathcal{T}_n\to \mathcal{R}_n$, which maps $\mathcal{T}_n^{\rm iso}$ into $\mathcal{R}_n^{\rm iso}$, the locus of isotrivial elliptic surfaces. It will then suffice to show that $\mathcal{R}_n^{\rm iso}$ is a closed algebraic subvariety of $\mathcal{R}_n$ with positive codimension (in fact, codimension at least $4$).

Now, every elliptic surface $f':\mathcal{S}'\to\mathbb{P}^1$ in $\mathcal{R}_n$ can be given in Weierstra\ss\ form:
\[
y^2 = x^3 + A(t)x + B(t),
\]
and if $f'$ is isotrivial then its Weierstra\ss\ form can be given explicitly, depending on the (constant) value of the $j$-invariant of the smooth fibers as follows:

If $j\neq 0, 12^3,$ then the Weierstra\ss\ form is
\begin{equation*}
	y^2 = x^3 + at^2x + bt^3,
\end{equation*}
with $a,b\in\mathbb{C}^\times$.
Indeed, the constancy of $j(t):=\tfrac{1728 \cdot 4A^3}{4A^3 + 27B^2}$ forces $A(t)\propto p(t)^2$ and $B(t)\propto p(t)^3$ for some rational function $p(t)$.
Because the underlying surface is rational and minimal, an Euler characteristic calculation shows that the fibration must have only two singular fibers of type $I_0^*$, so $p$ is of degree at most $1$, so by a M\"obius transformation we can assume $p(t)=t$.


If $j=0, 12^3$, then the special shape of the generic fiber
(admitting extra automorphisms) allows for explicit additional parameters as follows:

If $j=0$ then the Weierstra\ss\ form of $\mathcal S'$ is given by
\begin{equation*}
y^2 = x^3 + f(t),
\end{equation*}
where $f$ is a polynomial of degree at most $6$, which is not a perfect sixth power.

If $j=12^3$ then the Weierstra\ss\ form of $\mathcal S'$ is given by
\begin{equation*}
y^2 = x^3 - t(t-1)(t-\lambda)x.
\end{equation*}

According to our setup, we still have to account for the marked fibre,
but with all smooth fibres isomorphic, this simply amounts to fixing $q\neq 0,\infty$ (for $j\neq 0,12^3$)
resp.\ $f(q)\neq 0$ (for $j=0$) resp.\ $q\neq 0,1,\lambda,\infty$ (for $j=12^3$).

In all cases, the resulting parameter space $\mathcal{R}_n^{\rm iso}$ is Zariski closed in $\mathcal{R}_n$,
by inspection of \cite{Miranda} (cf. also the argument in \cite[Prop.6.1]{KLM}), enriched by the marking. Since $\mathcal{R}_n^{\rm iso}$ has dimension at most $4$,
it has positive codimension in $\mathcal{R}_n$ as claimed. Lastly, when $n=1$ there is no need to pass to the Jacobian surface since $f_p$ itself has a section, and the statement follows directly from the argument we just used.
\end{proof}


\subsection{The main result}
We now have all the ingredients to prove our main \autoref{main} and its \autoref{coro}.
\subsubsection{Proof of \autoref{main}}
 Let $E\subset\mathbb{P}^2$ be smooth cubic curve. Thanks to \autoref{dense} we know that $\cT_{\rm ni}$ is dense in $E^{(9)}$. Let $E^{(9)}\times \bP^2$ denote the trivial family given by the fixed surface $\bP^2$ and let $\cP\subset E^{(9)}\times \bP^2$ denote the tautological family of points, whose fiber over the $9$-tuple $p\in E^{(9)}$ consists of the $9$ points in $\bP^2$.
This is a codimension $2$ submanifold.
Blow it up and set
\[
	\cS:=Bl_\cP (E^{(9)}\times \bP^2) \text{ equipped with maps }\pi \colon \cS \to \bP^2\text{ and }p:\cS\to E^{(9)},
\]
and let $\iota:\cE\hookrightarrow \cS$ be the inclusion of the strict transform of $E^{(9)}\times E$. We also have the natural map $h:\cE\to E^{(9)}$. The notation is consistent with that of \autoref{sssec:setup_the_universal_family} if we denote by $\cS_p$, respectively $\cE_p$, the fibers of $\cS$ and $\cE$ over $p\in E^{(9)}$.

The normal bundle of $\cE$ inside $\cS$ is then given by
\[
	N_{\cE/\cS} = (\pi^* K_{\bP^2}
+\mathcal{F})^{\dual}\vert_\cE,
\]
where $\cF\subset \cS$ is the exceptional divisor of the blowup.

We may assume without loss of generality that $E^{(9)}$ is just a ball $B\subset\mathbb{C}^9$. This setup then matches precisely the assumptions in \autoref{siu}, and so the results in sections \ref{sec2}, \ref{sec3} and \ref{sec4} give us an upper semicontinuous radius function $\rho:B\to[0,+\infty]$ such that if $p\in B$ then $\rho(p)>0$ implies that the embedding $\cE_p\subset \cS_p$  is linearizable, with the converse also holding if $p\in B\backslash\mathcal{T}$.

As we have seen in \autoref{isotr}, $\rho|_{\cT_{\rm ni}}\equiv 0$, so the set $\{\rho=0\}$ is dense in $B$, and since $\rho$ is upper semicontinuous, it is a $G_\delta$ set. This means that $\{\rho>0\}$ is an $F_\sigma$ set with empty interior.

On the other hand, for each $n\geq 1$ the set $\cT_n$ is closed and with empty interior, so by Baire $\cT$ is an $F_\sigma$ set with empty interior. Hence
$\{\rho>0\}\cup \cT$ is also $F_\sigma$ with empty interior, and its complement $\{\rho=0\}\backslash \cT$ is thus a dense $G_\delta$, as desired.

\subsubsection{Proof of \autoref{coro}}
Let $E\subset\mathbb{P}^2$ be a smooth cubic with a configuration of $9$ distinct points to which \autoref{main} applies. Thus, the resulting embedding $C\subset S$ of the strict transform of $E$ into the rational surface $S$ has non-torsion normal bundle $N_{C/S}$, hence the embedding is formally linearizable by Arnold's \autoref{formal}, but it is not linearizable.

Let $S'=\mathbb{P}(N_{C/S}\oplus\mathcal{O})$ be the projective compactification of the total space of $N_{C/S}$ and $C'\subset S'$ the embedding of the zero section. Then the germs of embeddings $C\subset S$ and $C'\subset S'$ are formally isomorphic at the level of formal completions, but are not biholomorphic, as desired.


					\section{Complements and Questions}
					\label{sec:Complements-and-Questions}

\subsection{Brjuno's condition}
Given a smooth cubic curve $E\subset\mathbb{P}^2$, a line bundle $L\in\mathrm{Pic}^0(E)$ and a translation-invariant distance $d$ on $\mathrm{Pic}^0(E)$, define the small divisors
\begin{equation*}
	\Omega_n(L)=\min_{1\leq \ell\leq n}d(\mathcal{O}_E,L^{\ell}).
\end{equation*}
\begin{question}
Is it true that the Brjuno condition
\begin{equation*}
	\sum_{k=1}^\infty \frac{-\log\Omega_{2^k}(L)}{2^k}<\infty
\end{equation*}
holds for $L$ if and only if {\em every} configuration of $9$ distinct points on $E$ which has $N_{C/S}=L$ is linearizable?
\end{question}
Here we use the same notation as before: given $9$ distinct points on $E$, the blowup of $\mathbb{P}^2$ at these $9$ points is denoted by $S$, and the strict transform of $E$ is denoted by $C$.

The ``only if'' direction is a theorem of Ogawa \cite{ogawa}, while the ``if'' direction would be the analog of Yoccoz's Theorem \cite{yoccoz} in the Siegel linearization problem. This question is consistent with our \autoref{main}, since the set of elements in $\mathrm{Pic}^0(E)$ for which Brjuno's condition fails is a dense $G_\delta$.


	\subsection{Materialization of resonances}
		\label{ssec:Materialization-of-resonances}

\subsubsection{Materialization of Resonances}
	\label{sssec:Materialization-of-Resonances}
In the study of the single-variable Siegel linearization problem, a landmark theorem of Yoccoz \cite{yoccoz} shows that if the Brjuno condition fails for a given multiplier $\lambda\in S^1$, then the quadratic polynomial $f(z)=\lambda z+z^2$ is not linearizable, and furthermore there exists a sequence of periodic orbits approaching the fixed point at the origin.
Such a phenomenon is often referred to as ``materialization of resonances'', see e.g.  \cite{IP2}, \cite[\S 36.N]{arnold_book}.

In the setting of a germ of an elliptic curve in a surface $C\subset S$, a materialization of resonances would be the following (see \cite[\S 4.4]{arnold}).
In every neighborhood of $C$ in $S$, there exists another elliptic curve $C'$, which topologically (i.e. in a smooth tubular neighborhood) projects to an $n$-fold cover of $C$ for some $n\geq 1$.
Clearly this obstructs linearizability, since if an open neighborhood of $C$ in $S$ were biholomorphic to a neighborhood of the zero section in $N_{C/S}$, the elliptic curve $C'$ would give a section of $N^{\otimes n}_{C/S}$, which is impossible since $N_{C/S}$ has degree zero and is not torsion.

\subsubsection{No materialization of resonances in projective case}
	\label{sssec:No-materialization-of-resonances-in-projective-case}
The theorem of Yoccoz says that materialization of resonances explains all failures of linearizability in the family of quadratic polynomials.
One can ask if the same holds in the case of elliptic curves on surfaces, i.e. whether the materialization of resonances always happens when the embedding is not linearizable and the normal bundle is nontorsion.

In the case of a very general germ of a surface along an elliptic curve with non-torsion normal bundle, one can give a negative answer by using fixed points of holomorphic maps of the form $f(z)=\lambda z+O(z^2)$ with $\lambda$ on the unit circle which are known to be non-linearizable but with no materialization of resonances, via a standard suspension construction.
Such maps are constructed, for example, in \cite[Thm.~2]{Perez-Marco1993_Sur-les-dynamiques-holomorphes-non-linearisables-et-une-conjecture-de-V.-I.-Arnold}.
Nonetheless, such examples are expected to be exceptional (see \cite[1989-11]{Arnold2004_Arnolds-problems} and its comments in \S3), and indeed do not appear among quadratic polynomials by Yoccoz's theorem.
In fact, it appears to be open (see \cite[Thm.~11.12]{Milnor2006_Dynamics-in-one-complex-variable} and subsequent discussion) whether there exist rational maps of $\bP^1$ which are not linearizable at a fixed point and with no materialization of resonances.

Let us now explain why materialization of resonances \emph{never} occurs when $S$ is the blowup of $\bP^2$ at $9$ distinct points on a smooth cubic curve $E$ (with nontorsion normal bundle).
Indeed, suppose $C'$ is an elliptic curve in $S$ which projects to an $n$-fold cover of $C$ in a smooth tubular neighborhood of $C$ in $S$.
Denote by $p\colon S\to \bP^2$ the projection map.
Then $E':=p(C')$ is a plane curve of degree $3n$, which passes through the $9$ base points with multiplicity $n$.
The linear pencil generated by $nE$ and $E'$ is a Halphen pencil of index $n$, therefore the normal bundle of $E$ is torsion, which is a contradiction (see \autoref{neeman}).



	\subsection{Complete Calabi-Yau metrics}
		\label{ssec:Complete Calabi-Yau metrics}

Let $E\subset\bP^2$ be a smooth cubic, $S\to\bP^2$ the blowup at $9$ distinct points on $E$, and $C\subset S$ the strict transform of $E$. An interesting question to ask is:
\begin{question}
Does $S\backslash C$ always admit a complete Ricci-flat K\"ahler metric?
\end{question}

In the case when $N_{C/S}$ is trivial, such metrics exist thanks to the work of Hein \cite{hein}, which extended earlier work of Tian-Yau \cite{TY}. When $N_{C/S}$ is torsion but nontrivial, the same result holds by the work of Haskins-Hein-Nordstr\"om \cite{HHN}. It seems possible that the techniques in \cite{HHN} might also apply to the case when $N_{C/S}$ is nontorsion and linearizable, to show that in this case $S\backslash C$ also admits a complete Calabi-Yau metric. On the other hand, the last remaining case when $N_{C/S}$ is nontorsion and not linearizable (of which we now have many examples), appears to be very subtle and worthy of further investigation.



\bibliographystyle{sfilip_bibstyle}
\bibliography{linearization}

\end{document}